\documentclass[12pt,leqno,twoside,a4paper]{article}
\usepackage{amsfonts,amsmath,amsthm,amssymb}
\usepackage[english]{babel}
\usepackage{color}
\newtheorem{theorem}{Theorem}[section]
\newtheorem{proposition}[theorem]{Proposition}
\newtheorem{lemma}[theorem]{Lemma}
\newtheorem{corollary}[theorem]{Corollary}
\theoremstyle{definition}

\theoremstyle{remark}
\newtheorem{remark}[theorem]{\bf Remark}
\newtheorem{remarks}[theorem]{\bf Remarks}

 \numberwithin{equation}{section}

\theoremstyle{definition}
\theoremstyle{remark}
\numberwithin{equation}{section}

\begin{document}

\title{ ADS modules}
\date{}
\author{{\bf \normalsize Adel Alahmadi\footnote{This research was supported by the Deanship of Scientific Research, King Abdulaziz University, Jeddah, project no. 514/130/1432.}, S. K. Jain and Andr\'{e} Leroy} }
\vspace{6pt}

\maketitle\markboth{\rm  A. Alahmadi S.K. Jain and A. Leroy}{ \rm ADS modules}

\begin{abstract}
We study the class of $ADS$ rings and modules introduced by Fuchs
 \cite{F}.  We give some connections between this notion and 
classical notions such as injectivity and quasi-continuity.
A simple ring $R$ such that $R_R$ is ADS must be either right 
self-injective or indecomposable as a right $R$-module.  
Under certain conditions we can construct a unique ADS hull up to isomorphism.
We introduce the concept of completely ADS 
modules and characterize completely ADS semiperfect right modules
as direct sum of semisimple and local modules.
\end{abstract}

\section{INTRODUCTION}

The purpose of this note is to study the class of $ADS$ rings and modules.
Fuchs \cite{F} calls a right module $M$ right $ADS$ if for every
decomposition $M=S\oplus T$ of $M$ and every complement $T^{\prime }$ of $S$
we have $M=S\oplus T^{\prime }$. Clearly any ring in which idempotents are
central (in particular commutative rings or reduced rings) 
has the property that $R_{R}$ is $ADS$. 
Moreover, if $R$ is commutative then every cyclic $R$-module is $ADS$.
We note that every right quasi-continuous module (also known as $\pi$-injective module) is right $ADS$, but not conversely. However, a right ADS 
module which is also CS is quasi-continuous. We provide equivalent
conditions for a module to be $ADS$. A module need not have an $ADS$\ hull in
the usual sense but we show that, under some hypotheses, every nonsingular right module possesses a right $%
ADS$ hull which is unique up to isomorphism. We call a right module $M$ completely $ADS$ if each of its subfactors is $ADS$. We characterize
completely $ADS$ semiperfect right modules as direct sums of semisimple and
local modules. In particular we give an alternative proof of the characterizations of 
semiperfect $\pi c$-rings (rings whose cyclics are quasi-continuous).

\section{Definitions and Notations}

\bigskip Throughout every module will be a right module unless otherwise stated. All
rings have identity and all modules are unital. A module $M$ is called \textit{%
continuous} if it satisfies (C1): each complement in $M$ is a direct summand, and (C2):
if a submodule $N$ of $M$ is isomorphic to a direct summand of $M$ then $N$
itself is a direct summand of $M$. A module $M$ is called \textit{quasi-continuous} ($%
\pi $-injective) if it satisfies (C1) and (C3): the sum of two
direct summands of $M$ with zero intersection is again a direct summand of $M$.  Equivalently 
a module $M$ is quasi-continuous if and only if every projection $\pi_i:N_1\oplus N_2 \longrightarrow N_i$, where $N_i$ ($i=1,2$) are 
submodules of $M$, can be extended to $M$. 

For two modules $A$ and $B$, we say that $A$\textit{\ is }$B$\textit{%
-injective} if any homomorphism from a submodule $C$ of $B$ to $A$ can be extended to a homomorphism 
from $B$ to $A$. 
We note that if $A$ is 
$B$-injective and $A$ is contained in $B$ then $A$ is a direct summand of $B$.
A module $M$ is called \textit{semiperfect} if each of its homomorphic
images has a projective cover. A submodule $N$ of a module $M$ is \textit{%
small} in $M$ if for any proper submodule $P$ of $M$, $P+N\neq M$. We will
write $N<<M$.  
Let $A$ and $P$ be submodules of a module $M$.  Then $P$ is called a supplement of $A$ if it
is minimal with the property $A + P = M$. 



A module $M$ is \textit{discrete} if it satisfies (D$_{1}$): for every
submodule $A$ of $M$ there exists a decomposition $M=M_{1}\oplus M_{2}$ such
that $M_{1}\subset A$ and $M_{2}\cap A$ is small in $M$, and (D$_{2}$): if $A $ is a submodule of 
$M$ such that $M/A$ is isomorphic to a direct
summand of $M$, then $A$ is a direct summand of $M$. A module $M$ is called 
\textit{quasi-discrete} if it satisfies D$_{1}$ and D$_3$: if $M_1$ 
and $M_2$ are summands of $M$ and $M=M_1+M_2$ then $M_1\cap M_2$ is a 
summand of $M$.

\bigskip For any module $M$, $E(M)$ denotes the injective hull of $M$.  We
 recall a useful result of Azumaya that for any two modules $M$ and $N$,
if $M$ is $N$-injective then for any $R$-homomorphism $\sigma :E(N) \rightarrow E(M)$, 
$\sigma (N)\subseteq M$.

\bigskip 

\section{ PROPERTIES OF ADS MODULES}

\bigskip We begin with a lemma which is useful in checking the ADS property
of a module.  This was proved by  Burgess and Raphael \cite{BR}, however, for the sake of
completeness, we provide the proof.

\begin{lemma}
\label{ADS iff relative injectivity} An $R$-module $M$ is ADS if and only
if for each decomposition $M=A\oplus B$, $A$ and $B$ are mutually injective.
\end{lemma}

\begin{proof}
Suppose $M$ is $ADS.$ We prove $A$ is $B$-injective. Let $C$ be a submodule
of $B$ and let $f:C\rightarrow A$ be an $R$-$\hom $omorphism.   Set $X$=\{$%
c+f(c)\mid $ $c\in C\}$. Then $X\cap A=0$.  So $X$ is contained in a
complement, say $K,$ of $A$.  Then by hypothesis, $M=A\oplus K$.  The trick is to
define an $R$ -homomorphism $g:B\rightarrow A$ which is a
composition of the projection $\pi _{K}:M\rightarrow K$ along $A$ followed by
the projection $\pi _{A}:M\rightarrow A$ along $B$ and restricting to $B.$ By
writing an element $c\in C$ as $c=(c+f(c))-f(c)$, we see that $\pi _{A}\pi
_{K}=f$ on $C$ and hence $\pi _{A}\pi _{K}$ is an extension of $f.$

\noindent Conversely, suppose for each decomposition  $M=A\oplus B$, $A$
and $B$ are mutually injective. Let $C$ be a complement of $A$.  Set $U=B\cap
(A\oplus C)$ which is nonzero because $A\oplus C$ is essential in $M$.  Let  
$\pi _{A}$ be the projection of $A\oplus C$ on to $A$ and $f:U\rightarrow A$
be the restriction of  $\pi _{A}$ to $U$.  This can be extended \ to $%
g:B\rightarrow A,$ by assumption.  Let $b\in B$ and let $D=(b-g(b))R+C.$ We
claim D$\cap A=0$. Let $a\in A$ and let $a=br-g(b)r+c$ for some $c\in C$.
This gives $br=a+g(b)r-c\in U$ and so $f(br)=a+g(b)r$ because $f$ is
the identity on $A$ and $0$ on $C$.  This yields $a=0$, proving our claim.  Thus $D=C$
and hence $b-g(b)\in C$ for all $b\in B$.  Therefore, $b=(b-g(b))+g(b)\in
C\oplus A$ and so $M=A\oplus B\sqsubseteq C\oplus A,$ proving that $%
M=C\oplus A.$\end{proof}

Our next proposition gives equivalent statements as to when a module is ADS
analogous to characterization of quasi-continuous modules (Cf. \cite{GJ}). 

\begin{proposition}
\label{equivalent conditions for ADS} For an $R$-module $M$ the following are
equivalent:

\begin{enumerate}
\item[(i)] $M$ is ADS.

\item[(ii)] For any direct summand $S_{1}$ and a submodule $S_{2}$ having
zero intersection with $S_{1}$, the projection map $\pi _{i}:S_{1}\oplus
S_{2}\longrightarrow S_{i}$ ($i=1$, $2$) can be extended to an endomorphism
(indeed a projection) of $M$.

\item[(iii)] If $M=M_{1}\oplus M_{2}$ then $M_{1}$ and $M_{2}$ are mutually
injective.

\item[(iv)] For any decomposition $M=A\oplus B$, the projection $\pi
_{B}:M\longrightarrow B$ is an isomorphism when it is restricted to any
complement $C$ of $A$ in $M$,

\item[(v)] For any decomposition $M=A\oplus B$ and any $b\in B$, $A$
is $bR$-injective,

\item[(vi)] For any direct summand $A\subseteq^{\oplus} M$ and any $c\in M$ such that
$A\cap cR=0$, $A$ is $cR$-injective.
\end{enumerate}
\end{proposition}

\begin{proof}
(i)$\Rightarrow $ (ii) Let ${\hat{S}_{2}}$ be \ a complement of $S_{1}$
containing $S_{2}$. Then by definition of ADS module, $M=S_{1}\oplus {\hat{S}%
_{2}}$. Hence the canonical projections ${\hat{\pi _{1}}}:S_{1}\oplus {\hat{{%
S}_{2}}}\longrightarrow S_{1}$ and ${\hat{\pi _{2}}}:S_{1}\oplus {\hat{S}_{2}%
}\longrightarrow {\hat{{S}_{2}}}$ are clearly extensions of $\pi _{1}$ and $%
\pi _{2}$. 

\noindent (ii)$\Rightarrow $ (i) Let $M=A\oplus B$ and let $C$ be a
complement of $A$ in $M$. We must show that $M=A\oplus C$. By hypothesis,
the projection $\pi :A\oplus C\longrightarrow C$ can be extended to an
endomorphism $f:$ $M\longrightarrow M$. We claim $f(M)\subset C$. Since $%
A\oplus C$ is essential in $M$, for any $0\neq m\in M$, there exists an
essential right ideal $E$ of $R$ such that $0\neq mE\subset A\oplus C$. This
gives $f(m)E=\pi (mE)\subset C$. Since $C$ is closed in $M$, this yields $%
f(m)\in C$, proving our claim. We also remark that $f^{2}=f$, $%
M=Ker(f)\oplus im(f)$ and $Ker(f)=\{m-f(m)\;|\;m\in M\}$. We now show that 
$Ker(f)=A$. For any $a\in A$, clearly $a=a-f(a)\in Ker(f)$, hence $A\subset
Ker(f)$. Now let $0\neq $ $m-f(m)\in Ker(f)$. There exists $r\in R$ such
that $0\neq (m-f(m))r\in A\oplus C$. This implies $%
f[(m-f(m))r]=f(mr)-f(f(mr))=f(mr)-f(mr)=0$. Since $f$ extends $\pi $, this
means that $0\neq (m-f(m))r\in Ker(\pi )=A.$ But $A$ being closed in $M$, we
conclude $A=Ker(f)$, completing the proof. 

\noindent (i)$\Leftrightarrow $(iii) This
is Lemma \ref{ADS iff relative injectivity} above. 

\noindent (i)$\Leftrightarrow $(iv) Let $C$ be a complement of $A$. 
Then $\ker (\pi _{B}|_{C})=0$. Since 
$A\oplus C=(A\oplus C)\cap (A\oplus B)=((A\oplus C)\cap B)+A$, 
we have $\pi_{B}(C)=\pi _{B}(A\oplus C)=\pi _{B}((A\oplus C)\cap B)=(A\oplus C)\cap B$.
This gives $\pi _{B}(C)=B$ when $M$ is ADS. On the other hand if $\pi
_{B}(C)=B$ then $M=A\oplus C$, hence $M$ is ADS.

\noindent (i)$\Leftrightarrow $(v) This is classical (Cf. Proposition 1.4 in \cite{MM}).

\noindent (i)$\Rightarrow $(vi) Consider $C$ a complement of $A$ containing $cR$.
Since $M$ is ADS we have $M=A\oplus C$.  Using $(v)$, this leads to $A$ being $cR$-injective. 

\noindent (vi)$\Rightarrow $(i) This is clear since if $M=A\oplus B$, $(vi)$ implies that $A$ is $bR$-injective for all $b\in B$ and Proposition 1.4 in \cite{MM} yields that $A$ is $B$-injective.
\end{proof}

Let us mention the following necessary condition for a module to be ADS.

\begin{corollary}
Let $M_R$ be an ADS module.  For any direct summand $A\subseteq^{\oplus} M$ and any 
$(a,c,r)\in A\times M \times R$ such that $cR\cap A=0$ and $ann(cr) \subseteq r.ann(a)$ 
there exists $a'\in A$ such that $a=a'r$.
If $R$ is a right PID the converse is true.
\end{corollary}
\begin{proof}
By Proposition \ref{equivalent conditions for ADS}(vi), we know that $A$ is $cR$-injective.  
Consider $\varphi\in Hom_R(crR,A)$ defined by $\varphi(cr)=a$.  The condition on annihilators guarantees that $\varphi$ is well defined.  By relative injectivity, this map 
can be extended to $\overline{\varphi}:cR\longrightarrow A$, and hence we get $a=\varphi(cr)=\overline{\varphi}(c)r$.  We obtain the desired result by defining $a'=\overline{\varphi}(c)$.

\noindent If $R$ is a principal ideal domain then the submodules of $cR$ are of the form $crR$ for some 
$r\in R$.  The condition mentioned in the statement of the corollary makes it possible to extend any 
map in $Hom_R(crR,A)$ to a map in $Hom_R(cR,A)$ for any direct summand $A\subseteq ^\oplus M$.    
Invoking Proposition \ref{equivalent conditions for ADS}(vi), we can thus conclude that $M$ is ADS.
\end{proof}

\bigskip It is known that the sum of two closed submodules of a quasi-continuous 
module is closed \cite{GJ}. We prove that the
direct sum of two closed submodules of an ADS module is again closed when
one of them is a summand.

\begin{proposition}
Let $A$ and $B$ be two closed submodules of an ADS module $M$ such that $A$
is a summand and $A\cap B=0$. Then $A\oplus B$ is a closed submodule of $M$.
\end{proposition}

\begin{proof}
Let $C$ be a complement of $A$ containing $B$. Since $M$ is ADS, we have $%
M=A\oplus C$. Let $x=a+c$ be in the closure of $A\oplus B$ in $M$, where $%
a\in A$ and $c\in C$. Since $a\in A\subseteq cl(A\oplus B)$, we have that $%
a\in cl(A\oplus B)$. Hence there exists an essential right ideal $E$ of $R$
such that $cE \subseteq (A\oplus B) \cap C=B$. The fact that $B$ is closed
implies $c \in B$. Hence $x\in A\oplus B$, as desired.
\end{proof}

\begin{remark}
Let $A,B$ be closed submodules of an ADS module $M$ such that $A$ is a direct
summand of $M$. If $A\cap B$ is a direct summand of $M$, then $A + B$ is
closed. Indeed let $K$ be a complement of $A\cap B$. Since $M$ is ADS we
have $M=(A\cap B)\oplus K$. Hence $A+B=A\oplus (K\cap B)$. The above
proposition then yields the result.
\end{remark}

The proposition that follows gives an interesting property of an ADS module.  The original statement is due to Gratzer and Schmidt
(cf. Theorem 9.6 in \cite{F}).  
We first prove the following lemma.

\begin{lemma}
\label{Decomposition of injective hull goes down} 
Let $M=B\oplus C$ be a decomposition of $M$ with projections $\beta :M\rightarrow B$, $\gamma :M\rightarrow C$. 
Then $M=B\oplus C_1$ if
and only if there exists $\theta \in End(M)$ such that $C_1=(\gamma -
\beta\theta\gamma)(M)$
\end{lemma}

\begin{proof}
Suppose that $M=B\oplus C_{1}$ with projections $\beta _{1}$ on $B$ and $%
\gamma_{1}$ on $C_{1}$. We will show that $\beta _{1}=\beta +\beta \theta
\gamma $ and $\gamma _{1}=\gamma -\beta \theta \gamma $ with $\theta =\gamma
-\gamma _{1}$. We have $B<\ker (\theta )$, so $\theta =\theta \beta +\theta
\gamma =\theta \gamma $.

\noindent If $m=b+c=b_{1}+c_{1}$, where $b,b_{1}\in B,c\in C,c_{1}\in C_{1}$.  
Then $\theta(m) =c-c_{1}=b_{1}-b\in B$. Thus $\beta \theta =\theta $. Hence $%
\gamma _{1}=\gamma -\theta =\gamma -\beta \theta \gamma $. Also $\beta
_{1}=1_{A}-\gamma _{1}=\beta +\gamma -\gamma _{1}=\beta +$ $\beta \theta
\gamma $.

Conversely, if $\beta _{1}$, $\gamma _{1}$ are defined as above, that is $%
\beta _{1}=\beta +\beta \theta \gamma $ and $\gamma _{1}=\gamma -\beta
\theta \gamma $ for any $\theta \in End(M)$, then $\beta _{1}+\gamma
_{1}=1_{A}$, $\beta _{1}^{2}=\beta _{1}$, $\gamma _{1}^{2}=\gamma _{1}$, $%
\beta _{1}\gamma _{1}=\gamma _{1}\beta _{1}=0$. Therefore, $M=\beta
_{1}M\oplus \gamma _{1}M$. Since $\beta_{1}(M)\subset B$ and $%
\beta_1(b)=\beta(b)=b$ for $b\in B$, we have $M=B\oplus (\gamma -
\beta\theta\gamma)(M)$, as required.
\end{proof}

Using the same notations as in the previous lemma we state 
the following corollary.

\begin{corollary}
\label{characterization of ADS following Fuchs}
A module $M$ is ADS if and only if for any decomposition
$M=B\oplus C$ the complements of $B$ in $M$ are all of the form 
$(\gamma-\beta\theta\gamma)(M)$ for some $\theta\in End(M)$.
\end{corollary}

\begin{proposition}
Let $M=B\oplus C$ be a decomposition of an ADS $R$-module $M$. Let $\beta $ and $%
\gamma $ be projections on $B$ and $C$ respectively. Then the intersection $%
D$ of all the complements of $B$ is the maximal fully invariant submodule of 
$M$ which has zero intersection with $B$.
\end{proposition}

\begin{proof}
Let $\theta \in End(M)$. Then $C_{1}=(\gamma -\beta \theta \gamma )(M)$
is again a complement of $B$. For $c\in D$ we have $(\gamma -\beta \theta
\gamma )(c)=c$ and $\gamma c=c$, because \thinspace $c\in C_{1}\cap C$.
Hence $\beta \theta c=0$ and $\theta c\in C$. This holds for all complements $%
C$, so $\theta c\in D$, so $D$ is fully invariant in $M$ with $D\cap B=0$.
On the other hand, assume $X$ is fully invariant with $X\cap B=0$. Since $%
M=B\oplus C$, and $\pi _{B}(X)\subseteq X$ and $\pi_{C}(X)\subseteq X$, this
leads to $X=(X\cap B)\oplus (X\cap C)=X\cap C$. Hence $X<C$. Since $M$ is
ADS this holds for any complement of $B$ in $M$, and hence $X\subseteq D$.
\end{proof}

\bigskip

It is known that an indecomposable regular ring which is right
continuous is right self-injective (cf. Corollary 13.20 in \cite{G} ).  The following theorem
is a generalization of this result for simple rings without the
assumption of regularity.  We may add that an indecomposable two-sided continuous
regular ring is simple (cf. \cite{G} Corollary 13.26).

\bigskip

\begin{theorem}
Let $R$ be an ADS simple ring. Then either $R_{R}$ is indecomposable or $R$
is a right self-injective regular ring.
\end{theorem}

\begin{proof}
Let $Q$ be the right maximal quotient of $R$ which is regular right 
self-injective.  Since $R$ is  right (left) nonsingular $E(R)=Q$.
Suppose $R$ is not right indecomposable and let $e$ be a nontrivial
idempotent. Then since $R$ is ADS $eR$ is $(1-e)R$-injective (cf.
Lemma \ref{ADS iff relative injectivity}).  Furthermore, 
since $Hom((1-e)Q,eQ)\cong eQ(1-e)$, 
$(eQ(1-e)(1-e)R\subseteq eR$.  Because $R$ is simple, $R=$ $%
R(1-e)R\subset Q(1-e)R$.  This yields, $1\in Q(1-e)R$ . Therefore $Q=Q(1-e)R$,
and so $eQ=eR$. Similarly $(1-e)Q=(1-e)R$ hence $R=Q$, i.e. $R$ is a right self-injective
regular ring.
\end{proof}

\begin{corollary}
A simple regular right continuous ring is right self-injective. 
\end{corollary}

\bigskip 

\section{ADS HULLS}

\bigskip We now proceed to construct an $ADS$\ hull of a nonsingular module.
Burgess and Raphael (cf. \cite{BR}) claimed that an example can be constructed of a finite
dimensional module over a finite dimensional algebra which has no $ADS$
hull. We show that, under some circumstances, such an ADS hull does exist.

\begin{lemma}
Suppose $M$ is nonsingular. Then $M$ is ADS iff for every decomposition $%
E(M)=E_{1}\oplus E_{2}$ where $E_{1}\cap M$ is a direct summand of $M$, then $%
M=(E_{1}\cap M)\oplus (E_{2}\cap M)$.
\end{lemma}

\begin{proof}
Suppose $M$ is ADS. We may write $M=(E_{1}\cap M)\oplus K$ where $K$ is a
complement of $E_{1}\cap M$. Let $e_{i}:(E_{1}\cap M)\oplus (E_{2}\cap
M)\longrightarrow E_{i}\cap M$ be the projection map.  Then by Proposition \ref{equivalent 
conditions for ADS}(ii) there exists $e_{i}^{\ast
}:M\longrightarrow M$ that extends $e_{i}$. Let $\pi _{i}:E_{1}\oplus
E_{2}\longrightarrow E_{i}$ be the natural projection. Since $E(M)$ is
injective we can further extends $e_{i}\ast $ to $e_{i}^{\ast \ast }\in
End(E(M))$. We claim $e_{i}^{\ast \ast }$ is an idempotent in $End(E(M)$.
Indeed let $x\neq 0$ be any element in $E(M)$ and $A$ an essential right
ideal of $R$ such that $0\neq xA\subseteq M$.  We have $(e_{i}^{\ast \ast
})^{2}(x)A=(e_{i}^{\ast \ast })^{2}(xA)=(e_{i}^{\ast })^{2}(xA))=e_{i}^{\ast
}(xA)=e_{i}^{\ast \ast }(xA)=e_{i}^{\ast \ast }(x)A$. This yields the claim,
since $M$ is nonsingular. Thus
 $e_{i}^{\ast \ast }(E(M))=\pi
_{i}(E(M))=E_{i}$. Now $M\subseteq _{e}E(M)=E_{1}\oplus E_{2}$ implies $%
E_{1}\cap M\subseteq _{e}(E_{1}\oplus E_{2})\cap E_{1}$. Similarly $%
E_{2}\cap M\subseteq _{e}E_{2}$ and so $e_{i}^{\ast \ast }=\pi _{i}$ on $%
E_{1}\cap M\oplus E_{2}\cap M\subseteq _{e}M\subseteq _{e}E(M)$. Since $M$
is nonsingular $e_{i}^{\ast \ast }=\pi _{i}$ on $E(M)$. In particular, $\pi
_{i}(M)\subseteq M$ and so $M=(\pi _{1}+\pi _{2})(M)\subseteq \pi _{1}(M)\oplus
\pi _{2}(M)\subset (E_{1}\cap M)\oplus (E_{2}\cap M)$.

Conversely, let $M=A\oplus B$ and $C$ be a complement of $A$.  We must show that
$M=A\oplus C$.  Since $A\oplus C <_e M$, we get $E(M)=E(A)\oplus E(C)$.  
Since both $A$ and $C$ are closed in $M$, we have $E(A)\cap M=A$ and $E(C)\cap M=C$.
 Since $A$ is a direct summand of $M$ we have, thanks to the hypothesis, 
$M=(E(A)\cap M)\oplus (E(C)\cap M)=A\oplus C$, as desired.
\end{proof}

\bigskip 

\begin{theorem}
\label{Characterization of ADS via idempotents} Let $M$ be a right 
$R$-module.  Then $M$ is ADS if and only if for every $e=e^2$, $f=f^2\in End(E(M))$
with $eM\subset M$ and $fE(M)= eE(M)$, we have $fM\subset M$.
\end{theorem}

\begin{proof}
Let us prove necessity: $(1-f)(E(M))\cap M \subseteq _e(1-f)(E(M))$ and $%
f(E(M))\cap M\subseteq _e f(E(M))$. Thus $((1-f)(E(M))\cap M )\oplus
(f(E(M))\cap M )\subseteq _e M$. We claim $f(E(M))\cap M=e(M)$. Note first that $e(E(M)) \cap
M= f(E(M))\cap M$. Clearly $eE(M)\cap M \subseteq eM$. Let $%
C=(1-f)(E(M))\cap M $.  Then 
$C\oplus eM \subseteq_e M$. Because $eM$ is closed $C$ is a complement of $%
eM$ in $M$ (cf. Lemma 6.32 in Lam's book). Because $M$ is ADS we have $%
M=e(M) \oplus C$. Let $g$ be the projection of $eM$ along $C$, so that $%
g(M)=e(M)$. Now $g(M)=e(M)\subseteq f(E(M))$. This gives $eM=(M)=fg(M)=fe(M)$%
. Since $C$ is contained in $(1-f)(E(M))$, $f(C)=0$. Then $fM =f(C\oplus
eM)=eM\subseteq M$.

\noindent Conversely, let $M=eM\oplus (1-e)(M)$ and $C$ be a complement of $%
e(M)$ in $M$. We want to show $M=e(M)\oplus C$. Now, $C\oplus e(M)\subseteq
_{e}M$ and so $E(C)\oplus E(eM)=E(M)$. Hence $E(C)\oplus eE(M)=E(M)$. Let $f$
be the projection on $eE(M)$ along $E(C)$. We have $f(E(M))=e(E(M))$ and 
$E(C)=(1-f)(E(M))$. By hypothesis we have $f(M)\subseteq M$. Let $m$ be in $M$.
  Then $m\in M=E(C)\oplus f(E(M))$, say $m=c+f(m)$, where $c\in E(C)$. 
$c=m-f(m)\in E(C)\cap M=C$, because $C$ is closed. We conclude that 
$M=C\oplus e(M)$.
\end{proof}

We may recall that any endomorphism $f\in End_R(M)$ of a nonsingular module $M$ can be uniquely 
extended to an endomorphism $f^*$ of 
its injective hull $E(M)$.  Let us mention moreover that if $f=f^2$ then $f^*=(f^*)^2$.
Under these notations we obtain the following corollary.
\begin{corollary}
Let $M$ be a right nonsingular 
$R$-module. $M$ is ADS if and only if for every $e=e^2\in End(M)$ and $f=f^2\in End(E(M))$
with $fE(M)= e^*E(M)$, we have $fM\subset M$.
\end{corollary}
We are now ready to show, that under some circumstances, an ADS hull can be constructed for a 
nonsingular module.  For a nonsingular right $R$-module $M$, we continue to let $e^*$ denote the unique 
extension of $e^2=e\in End(M)$ to the injective hull $E(M)$ of $M$.

\begin{theorem}
\label{ADS hull} Let $M_R$ be a nonsingular right $R$-module. Let  
$\overline{M}$ denote the intersection of all the ADS submodules of $E(M)$ containing $M$.   
Suppose that for any $e^2=e\in End(\overline{M})$ and for any ADS submodule $N$ of $E(M)$ 
containing $M$ we have $e^*(N) \subset N$.  Then, $\overline{M}$ is, up
to isomorphism, the unique ADS hull of $M$.
\end{theorem}
\begin{proof}
Let $\Omega$ be the set of ADS submodules $N$ such that 
$M<N<E(M)$.
Then $\overline{M}=\bigcap_{N\in \Omega }N$. We claim that $\overline{M}$ is
ADS. Clearly $E(\overline{M})=E(M)$. Let $e=e^2\in End_R(M),f^2=f\in End(E(M))$
such that $e(\overline{M})\subseteq \overline{M}$ and $f(E(M))=e^*(E(M))$.
Since $M$ is nonsingular and $e(\overline{M})\subseteq \overline{M}$, we have 
$e(N)\subseteq N$ for every $N\in \Omega $. So, for every $N\in \Omega $, $%
f(N)\subset N$ because $N$ is ADS. Let $x\in \overline{M}$.  Then $x\in N$ for
every $N\in \Omega $. Hence $f(x)\in N$ for every $N\in \Omega $. Therefore $%
f(x)\in \bigcap_{N\in \Omega }N=\overline{M}$, that is $f(\overline{M}%
)\subseteq \overline{M}$, proving our claim.
\end{proof}

\begin{remarks}
Let us remark that the condition stated in the above theorem is in particular fulfilled if we
consider the ADS hull of a nonsingular ring.  Indeed in this case we consider the ADS rings between
$R$ and $Q:=E(R)$ and projections are identified with idempotents of the rings.  Of course, these 
idempotents remain idempotents in overrings.
\end{remarks}

\bigskip 

\section{COMPLETELY $ADS$ MODULES  }

\begin{theorem}
\label{General decomposition} Let $M=\oplus _{i\in I}M_{i}$ be a
decomposition of a module $M$ into a direct sum of indecomposable modules $%
M_{i}$. Suppose $M$ is completely ADS. Then

\begin{enumerate}
\item[(i)] For every $(i$, $j)\in I^{2}$, $i\ne j$, $M_{i}$ is $M_{j}$-injective.

\item[(ii)] If $(i,j)\in I^{2}$, $i\ne j$ are such that $Hom_{R}(M_{i}$, $%
M_{j})\neq 0$, then $M_{j}$ is simple.

\item[(iii)] $M=S\oplus T$ where $S$ is semisimple and $T=\oplus _{j\in
J\subset I}M_{j} $ is a direct sum of indecomposable modules. Moreover, for
any $\theta \in End(M)$ we have $\theta(S)\subset S$ and  for $j\in J$, 
$\theta(M_{j})\subseteq M_{j}\oplus S$.  
\end{enumerate}
\end{theorem}

\begin{proof}
Since the ADS property is inherited by direct summands, statement $(i)$ is
an obvious consequence of Lemma \ref{ADS iff relative injectivity}.

\noindent $(ii)$ For convenience, let us write $i=1$, $j=2$ and suppose that 
$0\ne \sigma \in Hom_R(M_1,M_2)$. We have $\sigma(M_1)\oplus M_2\oplus
\dots\cong M_1/\ker(\sigma)\oplus M_2\oplus \dots =M/\ker(\sigma)$ is ADS, by
assumption. Hence $\sigma(M_1)$ is $M_2$-injective and, since $\sigma(M_1)
\subseteq M_2$, we get that $\sigma(M_1)$ is a direct summand of $M_2$. But $%
M_2$ is indecomposable, hence $\sigma(M_1)=M_2$. 
We conclude that $M_2\oplus M_2=\sigma(M_1)\oplus M_2$ is $ADS$. 
This means that $M_2$ is 
$M_2$-injective i.e. $M_2$ is quasi-injective.

Let us now show that for any $0\ne m_2\in M_2$, $m_2R=M_2$. Since $%
\sigma(M_1)=M_2$, there exists $m_1\in M_1$ such that $\sigma(m_1)=m_2$. We
remark that $\sigma(m_1R)\oplus M_2=\frac{m_1R}{\ker \sigma\cap m_1R}\oplus
M_2=\frac{m_1R\oplus M_2}{\ker\sigma \cap m_1R}$ is a submodule of $\frac{M}{%
\ker \sigma \cap m_1R}$. Since $M$ is completely ADS, we conclude that $%
\sigma(m_1R)\oplus M_2$ is ADS. As earlier in this proof, relative
injectivity and indecomposability lead to $\sigma(m_1R)=M_2$. Hence $%
m_2R=M_2 $, as desired.

$(iii)$ Let $I_1$ consist of those $i\in I$ such that
there exists $j\in I$, $j\ne i$ with $Hom_R(M_j,M_i)\ne 0$.  We define $S:=\oplus_{i\in I_1}M_i$ and $T:=\oplus_{j\in J}M_j$ where $J:=I\setminus I_1$.  
Statement $(ii)$ above implies that $M=S \oplus T$ where $S$ is
semisimple and $T$ is a sum of indecomposable modules. Moreover if $j\in J$,
then for any $i\in I, \; i\ne j$, we have $Hom_R(M_i,M_j)=0$. It is clear that, for any 
$\theta \in End(M)$ we must have $\theta (S)\subset S$.  For $j\in J$ and $x\in M_j$ let us write
$\theta(x) = y + z$, where $z\in S$ and $y\in T$.  Since, for $l\in J$, $l\ne j$, $Hom_R(M_j,M_l)=0$, we have $\pi_l\theta(x)=0$, where $\pi_l: M\rightarrow M_l$ is the natural epimorphism.   Thus $\pi_l(y)=0$.  This shows that $y\in M_j$, as required.
\end{proof}

Oshiro's theorem states that any quasi-discrete module is a direct sum of indecomposable modules (cf. \cite{MM} Theorem $4.15$).  Hence the above Theorem \ref{General decomposition} applies to  completely ADS quasi-discrete modules.  In general for a quasi-discrete module we have the following theorem:
\begin{theorem}
Let $M$ be a completely ADS quasi-discrete module.  Then $M$ can be written as 
$M=S\oplus M_1\oplus M_2$, where $S$ is semisimple, $M_1$ is a direct sum of local modules and
$M_2$ is equal to its own radical.
\end{theorem}
\begin{proof}
Corollary 4.18 and Proposition 4.17 in \cite{MM} imply that $M=N\oplus M_2$ where $N$ has a small 
radical and $M_2$ is equal to its own radical.  Theorem \ref{General decomposition} applied to $N$
yields the conclusion.
\end{proof}
  
We now apply the previous theorem to the case of semiperfect modules.

\begin{theorem}
\label{semiperfect modules} Let $M$ be a semiperfect module with a completely ADS
projective cover $P$. Then $M$ can be presented as $M = S\oplus T$ where $S$
is semisimple and $T$ is a sum of local modules. Moreover any partial sum in
this decomposition contains a supplement of the remaining terms.
\end{theorem}
 
\begin{proof}

Clearly $P$ is semiperfect and projective (cf. Theorem $11.1.5$ in \cite{K}).  
Combining the statements in $42.5$ in \cite{W} and Corollary $4.54$ in \cite{MM},
we get that $P$ is discrete and is a direct sum of local modules.
The remark preceding the present theorem then implies that
 we can write $P=S^{\prime }\oplus T^{\prime }$ 
where $S^{\prime }$ is semisimple and 
$T^{\prime }$ is a direct sum of indecomposable local modules.  Let $\sigma$ be an
onto homomorphism from $P$ to $M$ with small kernel $K$. We thus have 
$M=\sigma(S^{\prime }) + \sigma(T^{\prime })$. Since homomorphic images of $M$
have projective covers, Lemma 4.40 \cite{MM} shows that $\sigma(T^{\prime })$
contains a supplement $X$ of $\sigma(S^{\prime })$. In particular, we have $%
\sigma(S^{\prime }) \cap X<< X$. Since $\sigma(S^{\prime })$ is semisimple we
conclude that $\sigma(S^{\prime }) \cap X=0$ and hence $M=\sigma(S^{\prime
}) \oplus \sigma(T^{\prime })$. Since homomorphic images of a local module
are still local, we conclude that the terms appearing in $\sigma (T^{\prime
}) $ are local modules. The last statement is a direct consequence of Lemma 4.40 \cite{MM}.
\end{proof}

Let us mention that local rings which are not uniform give examples of 
semiperfect completely ADS modules which are not CS and hence not quasi-continuous.  
 
The following corollary characterizes semiperfect $\pi c$-rings providing a
new proof of Theorem 2.4 in \cite{GJ}.

\begin{theorem}
Let $R$ be a semiperfect ring such that every cyclic module is quasi-continuous. 
Then $R=\oplus_{i\in I} A_i$ where each $A_i,\; i\in I$ is simple artinian or a
valuation ring.
\end{theorem}

\begin{proof}

Since $R$ is semiperfect $R=B_1\oplus B_2\oplus \dots \oplus B_n$ a direct 
sum of indecomposable right ideals.  In view of the fact that quasi-continuous modules are ADS,
Theorem \ref{General decomposition} gives a decomposition 
$R=e_1R\oplus e_2R \oplus \dots \oplus e_kR \oplus \dots \oplus e_nR$ where $e_iR$ are simple right ideals 
for $1\le i \le k$ and $e_jR$ are local right ideals for $k < j \le n$.
Let $\sigma$ be a homomorphism from $e_sR$ to $e_tR$ for some  $1\le s,t \le n$.
Then $e_sR/\ker \sigma$ embeds in  $e_tR$.  Since $R/\ker( \sigma)$ is quasi-continuous,
$e_sR/\ker \sigma$ is $e_tR$-injective and hence $e_sR/\ker(\sigma)$ splits in $e_tR$.  This shows that either
$e_sR/\ker(\sigma)\cong e_tR$ or $\ker(\sigma)=e_sR$, that is $\sigma=0$.  Since $e_tR$ is projective, if 
$e_sR/\ker(\sigma)\cong e_tR$, then $\ker(\sigma)$ splits in $e_sR$, thus $\ker(\sigma)=0$.  In short
we get that if $\sigma\ne 0$ then $e_sR\cong e_tR$, the latter isomorphism implies $e_sR$ and $e_tR$ are 
minimal right ideals (cf. Lemma 2.3 in \cite{GJ}).  By grouping the right ideals $e_iR$ according to their 
isomorphism classes, we get $R=A_1\oplus A_2\oplus \dots \oplus A_l$, $l\le n$, 
 where each $A_i$ is either a simple artinian ring
or a local ring. We claim that if $A_i$ is a local ring then it is a
valuation ring. We thus have to show that any pair of two nonzero submodules 
$C,D$ of the ring $A_i$ are comparable. Let us consider the right submodules 
$\frac{C}{C\cap D}$ and $\frac{D}{C\cap D}$ of $\frac {R}{C\cap D}$. Since $%
A_i/ (C \cap D)$ is a local quasi-continuous it is uniform, but $C / (C \cap D)
\bigcap D/ (C \cap D) =0$. Therefore $C/(C \cap D)$ or $D/(C \cap D) = 0$ hence $C$
and $D$ are indeed comparable.
\end{proof}

Let us conclude this paper with some questions:
\begin{enumerate}
\item It is known that if $R_R$ and $_RR$ are both CS then $R$ is Dedekind finite.
What could be the analogue of this for ADS modules?
\item Does a directly finite ADS module have the internal cancellation property?
(cf. Theorem 2.33 in \cite{MM}, for the quasi-continuous case).
\item  What can be said of a module which is ADS and has the C$_2$ property?
\end{enumerate}

\vspace{5mm}

\centerline{ACKNOWLEDGEMENT}

We thank the referee for drawing our attention to a number of  typos.

\vspace{5mm}
\noindent \normalsize Adel Alahmadi\\
\normalsize King Abdullaziz University, Jeddah, Saudi Arabia\\
    \normalsize E-mail: adelnife2@yahoo.com \vspace{6pt}\\
\normalsize S.K. Jain\\
\normalsize King Abdullaziz University, Jeddah, Saudi Arabia\\
\normalsize and Ohio University, Athens, USA \\
     \normalsize E-mail: jain@ohio.edu \vspace{6pt}\\
\normalsize  Andr\'e Leroy\\                                                                                                              
\normalsize  Universit\'{e} d'Artois,  Facult\'{e} Jean Perrin, Lens, France\\
\normalsize and MECAA, King Abdullaziz University, Jeddah, Saudi Arabia\\
 \normalsize  E-mail: leroy@euler.univ-artois.fr


\vspace{5mm}

\begin{thebibliography}{GJ}

\bibitem[BR]{BR} W.D. Burgess, R. Raphael, On modules with the absolute direct summand property, Ring 
theory (Granville, OH, 1992), 137–148, World Sci. Publ., River Edge, NJ, 1993.

\bibitem[F]{F} L. Fuchs, Infinite abelian groups, Pure and applied mathematics,
A Series of Monographs and Textbooks Vol 36, I, 
Academic Press, New York, San Francisco, London 1970.

\bibitem[G]{G} K.R. Goodearl, Von Neumann regular rings, Krieger publishing company, Malabar, Florida 1991. 

\bibitem[GJ]{GJ} V.~K.~Goel and S.~K.~Jain, $\pi$-injective modules and rings whose
cyclics are $\pi$-injective, Communications in Algebra, \textbf{6}, 1978,
59-73.

\bibitem[K]{K} F.~Kasch, Modules and rings, Academic Press, 1982

\bibitem[MM]{MM} S.H.~Mohammed and B.J.~M\"{u}ller, 
Continuous and discrete modules, London Mathematical Society Lecture
Notes Series, 147, Cambridge University Press, New York, 1990.

\bibitem[W]{W} R. Wisbauer, Foundations of Module and Ring Theory,
Gordon and Breach Science Publishers, Reading (1991).


\end{thebibliography}
\end{document}